\documentclass[12pt]{article}
\usepackage{a4}
\usepackage{amsthm}
\usepackage{amsfonts}
\usepackage{amssymb}
\usepackage{amsmath}
\usepackage{cite}
\usepackage{epsfig}
\usepackage{stmaryrd}
\usepackage{xcolor}
\newtheorem{theorem}{Theorem}
\newtheorem{lemma}[theorem]{Lemma}
\newtheorem{proposition}[theorem]{Proposition}
\newtheorem{corollary}[theorem]{Corollary}
\newcommand{\inv}{{\operatorname{inv}}}
\newcommand{\unorient}[1]{\widehat{#1}}
\newcommand{\sym}[1]{#1^{(\rm s)}}
\newcommand{\symc}[1]{\overline{#1}}
\newcommand\PP{{\cal P}}

\newcommand\NN{{\mathbb N}}

\DeclareTextCompositeCommand{\v}{OT1}{l}{l\nobreak\hspace{-.1em}'}
\DeclareTextCompositeCommand{\v}{OT1}{t}{t\nobreak\hspace{-.1em}'\nobreak\hspace{-.15em}}
\DeclareMathOperator{\sgn}{sgn}

\begin{document}
\title{Uniform Tur\'an density---palette classification\thanks{This research started during the first author's stay at HUN-REN Alfr\'ed R\'enyi Institute of Mathematics in the spring 2024, which was supported by the ERC Synergy Grant No. 810115 (DYNASNET). The first and fourth authors also acknowledge support by the National Science Foundation under Grant No. DMS-1928930 while the authors were in residence at the Simons Laufer Mathematical Sciences Institute (formerly MSRI) in Berkeley, California, during the Spring 2025 semester. The work presented in this manuscript appeared in the form of an extended abstract in the proceedings of Eurocomb'25.}}

\author{Daniel Kr{\'a}\v l\thanks{Institute of Mathematics, Leipzig University, Augustusplatz 10, 04109 Leipzig, and Max Planck Institute for Mathematics in the Sciences, Inselstra{\ss}e 22, 04103 Leipzig, Germany. E-mail: {\tt daniel.kral@uni-leipzig.de}. Previous affiliation: Faculty of Informatics, Masaryk University, Botanick\'a 68A, 602 00 Brno, Czech Republic.}\and
	Filip Ku\v cer\'ak\thanks{Faculty of Informatics, Masaryk University, Botanick\'a 68A, 602 00 Brno, Czech Republic. E-mail: {\tt 514391@fi.muni.cz}.}\and
        Ander Lamaison\thanks{Universidad P\'ublica de Navarra, Pamplona, Spain. Previous affiliation: Extremal Combinatorics and Probability Group (ECOPRO), Institute for Basic Science (IBS), Daejeon, South Korea. This author was supported by IBS-R029-C4. E-mail: {\tt ander.lamaison@unavarra.es}.}\and
        G\'abor Tardos\thanks{HUN-REN R\'enyi Institute, Re\'altanoda utca 13--15, Budapest, Hungary. This author was also supported by the ERC Advanced Grants ``ERMiD'' and ``GeoScape'' and by the National Research, Development and Innovation Office project K-132696. E-mail: {\tt tardos@renyi.hu}.}}

\date{}

\maketitle

\begin{abstract}
In the 1980s,
Erd\H{o}s and S\'os initiated the study of Tur\'an hypergraph problems with a uniformity condition on the distribution of edges,
i.e., determining density thresholds for the existence of a hypergraph $H$ in a host hypergraph with edges uniformly distributed.
In particular, Erd\H{o}s and S\'os asked to determine the uniform Tur\'an densities of the hypergraphs $K_4^{(3)-}$ and $K_4^{(3)}$.
After more than 30 years,
the former was solved by Glebov, Kr\'al' and Volec [Israel J. Math. 211 (2016), 349--366] and
Reiher, R\"odl and Schacht [J. Eur. Math. Soc. 20 (2018), 1139--1159],
while the latter still remains open.
In these two cases and several additional cases,
the tight lower bound is provided by a so-called palette construction,
i.e., the existence of a palette that $H$ is not colorable with and
that has the density equal to the uniform Tur\'an density of $H$.

Lamaison~[arXiv:2408.09643] has recently showed that the uniform Tur\'an density of a $3$-uniform hypergraph $H$
is equal to the supremum of the densities of palettes that $H$ is not colorable with.
We give a necessary and sufficient condition, which is easy to verify,
on the existence of a $3$-uniform hypergraph colorable by each of palettes $\PP_1,\ldots,\PP_r$ and
not colorable by any palette $\PP'_1,\ldots,\PP'_{r'}$.
We also demonstrate how our result
can be used to prove the existence of $3$-uniform hypergraphs with specific values of the uniform Tur\'an density.
\end{abstract}

\maketitle

\section{Introduction}
\label{sec:intro}

Tur\'an problems are one of the most fundamental problems in extremal combinatorics;
they ask to determine the minimum density, which is called \emph{Tur\'an density}, that
guarantees the existence of a given substructure.
The name originates from the classical theorem of Tur\'an~\cite{Tur41},
which determines Tur\'an densities for all complete graphs.
While Tur\'an densities of graphs are well-understood due to work of Erd\H os and Stone~\cite{ErdS46} from the 1940s,
also see~\cite{ErdS66},
the same is not the case for Tur\'an densities of hypergraphs,
where even some of the most basic problems have stayed open for many decades.
In particular, 
the Tetrahedron Problem, which asks for the Tur\'an density of $K_4^{(3)}$, the complete $3$-uniform hypergraph with $4$ vertices,
has defied resolution since its formulation over 80 years ago~\cite{Tur41}, and
determining the Tur\'an density of all complete $k$-uniform hypergraphs for $k\ge 3$
is a \$1\,000 problem of Erd\H os~\cite{Erd81} (Erd\H os also offered \$500 for determining the Tur\'an density of
any single complete $k$-uniform hypergraph with at least $k+1$ vertices).
We refer to~\cite{FraF84,ChuL99,Raz10} for additional results and
also to the surveys by Keevash~\cite{Kee11} and Sidorenko~\cite{Sid95} for a more comprehensive treatment of the matter.

Almost all known and conjectured extremal constructions for Tur\'an problems in the (hyper)graph setting have large independent sets,
i.e., the edges are spread in a highly non-uniform way in the host hypergraph.
This led Erd\H os and S\'os~\cite{ErdS82,Erd90} to studying Tur\'an densities
with the additional requirement that the edges of the host hypergraph are distributed uniformly.
Formally, we say that an $n$-vertex $k$-uniform hypergraph is $(d,\varepsilon)$-uniformly dense
if every subset of $n'\ge\varepsilon n$ vertices spans at least $d\binom{n'}{k}$ edges, and
the \emph{uniform Tur\'an density} of a $k$-uniform hypergraph $H$
is the infimum over all $d$ such that for every $\varepsilon>0$,
there exists $n_{\varepsilon}$ such that
every $k$-uniform hypergraph with $n\ge n_{\varepsilon}$ vertices that is $(d,\varepsilon)$-uniformly dense
contains $H$ as a subhypergraph.
We remark that the notion is interesting for $k$-uniform hypergraphs with $k\ge 3$ only,
since it can be shown that the uniform Tur\'an density of every graph is equal to zero.

Until about a decade ago, there was very little progress concerning the uniform Tur\'an densities of hypergraphs
since Erd\H os and S\'os introduced the notion and asked about determining the uniform Tur\'an densities of
the $3$-uniform hypergraphs $K_4^{(3)}$ (tetrahedron), which is the complete hypergraph on $4$ vertices, and
$K_4^{(3)-}$ (broken tetrahedron), which is $K_4^{(3)}$ with an edge removed.
The latter problem was resolved by Glebov, Volec and the first author~\cite{GleKV16}
using the flag algebra method of Razborov~\cite{Raz07}, and
a direct combinatorial argument using the hypergraph regularity method was given by Reiher, R\"odl and Schacht~\cite{ReiRS18a}.
The use of the hypergraph regularity method revolutionized the area.
Reiher, R\"odl and Schacht~\cite{ReiRS18} classified $3$-uniform hypergraphs with uniform Tur\'an density equal to $0$, and
showed that the uniform Tur\'an density of every $3$-uniform hypergraph is at least $1/27$ unless it is equal to $0$;
the value of $1/27$ was shown to be tight in~\cite{GarKL24}.
Families of $3$-uniform hypergraphs with uniform Tur\'an density equal to $1/27$, $4/27$, $1/4$ and $8/27$
were identified in~\cite{BucCKMM23,CheSXX,GarIKLXX,LiLWZ25}, and
very recently the uniform Tur\'an densities of all generalized stars, largely extending the result concerning $K_4^{(3)-}$,
have been determined by Wu and the third author~\cite{LamW24}.
For further exposition including results on stronger notions of uniform density such as e.g.~\cite{ReiRS16,ReiRS18b,ReiRS18c},
we refer the reader to the survey by Reiher~\cite{Rei20} on the topic.

In our further exposition, we restrict our attention to $3$-uniform hypergraphs only (while dropping the adjective).
All lower bounds on the uniform Tur\'an density of a hypergraph come from so-called palette constructions,
which extend the lower bound construction of R\"odl~\cite{Rod86} in the tetrahedron case.
A \emph{palette} $\PP$ is a pair $(C,T)$ with $T\subseteq C^3$;
we refer to the elements of $C$ as \emph{colors} and the elements of $T$ as (feasible) \emph{triples}.
The \emph{density} of a palette $\PP=(C,T)$, which is denoted by $d(\PP)$, is $|T|/|C|^3$.
Consider the following random construction of an $n$-vertex hypergraph $H_n$ with vertices $v_1,\ldots,v_n$:
color each pair $(v_i,v_j)$, $1\le i<j\le n$, uniformly at random with a color $c\in C$ and
include $\{v_i,v_j,v_k\}$, $1\le i<j<k\le n$, as an edge if $(x,y,z)\in T$
where $x$ is the color of $(v_i,v_j)$, $y$ is the color of $(v_i,v_k)$ and $z$ is the color of $(v_j,v_k)$.
Observe that $H_n$ is $(d(\PP)-\varepsilon,\varepsilon)$-uniformly dense with positive probability
when $n$ is sufficiently large (with respect to $\varepsilon>0$).

We say that a hypergraph $H$ is $\PP$-colorable
if there exists an order $v_1,\ldots,v_N$ of its vertices and coloring of their pairs with the colors from $C$ such that
every edge $\{v_i,v_j,v_k\}$ of $H$, $1\le i<j<k\le N$, satisfies that
$(x,y,z)\in T$ where $x$ is the color of $(v_i,v_j)$, $y$ is the color of $(v_i,v_k)$ and $z$ is the color of $(v_j,v_k)$.
Observe that if $H$ is not $\PP$-colorable for a palette $\PP$,
then $H$ is avoided by any $\PP$-colorable hypergraph.
In particular, the uniform Tur\'an density of $H$ is at least $d(\PP)$ as
$H$ is avoided by the random hypergraph $H_n$ constructed in the previous paragraph (with probability one).

We may generalize the construction of the hypergraph $H_n$
by coloring each pair of vertices with a color $c\in C$ with probability $p(c)$,
which is not necessarily equal to $1/|C|$.
This leads us to the definition of the \emph{Lagrangian} $L(\PP)$ of a palette $\PP$, which is
\[\max_{p:C\to [0,1], \sum p=1}\sum_{(x,y,z)\in T}p(x)p(y)p(z)\]
where the maximum is taken over all probability distributions $p$ on $C$,
i.e., functions $p:C\to [0,1]$ such that $\sum_{x\in C}p(x)=1$.
Observe that for every $\varepsilon>0$ and sufficiently many vertices,
the generalized construction yields a $\PP$-colorable $(L(\PP)-\varepsilon,\varepsilon)$-uniformly dense hypergraph with high probability.
Since a hypergraph $H$ that is not $\PP$-colorable
cannot be contained in such an $(L(\PP)-\varepsilon,\varepsilon)$-uniformly dense hypergraph,
the Lagrangian of a palette $\PP$ is a lower bound on the uniform Tur\'an density of $H$.
We formulate this conclusion as a proposition for future reference.

\begin{proposition}
\label{prop:Lagrangian}
Let $H$ be a $3$-uniform hypergraph and $\PP$ a palette.
If the hypergraph $H$ is not $\PP$-colorable,
then the uniform Tur\'an density of $H$ is at least the Lagrangian of $\PP$.
\end{proposition}

The empirical evidence was strongly suggesting that the palette constructions always provide tight lower bounds.
This has recently been proven by the third author~\cite{Lam24}
using a combination of the hypergraph regularity method and probabilistic arguments.

\begin{theorem}
\label{thm:palette}
Let $H$ be any $3$-uniform hypergraph.
The uniform Tur\'an density of $H$
is equal to the supremum of the Lagrangian of a palette $\PP$ such that $H$ is not $\PP$-colorable.
\end{theorem}

\noindent We remark that it is not hard to show that the Lagrangian of a palette
can be replaced with the density in the statement of Theorem~\ref{thm:palette}.

Theorem~\ref{thm:palette} presents a breakthrough in regard to the methodology for determining the uniform Tur\'an densities as
it permits replacing complex arguments involving the hypergraph regularity method
with much simpler arguments concerning colorability by palettes.
In particular,
Theorem~\ref{thm:palette} was essential in computing the uniform Tur\'an densities of generalized stars by Wu and the third author~\cite{LamW24},
a result vastly extending the solution of the problem of Erd\H os-S\'os
on the uniform Tur\'an density of $K_4^{(3)-}$.

Given that the colorability of palettes with certain densities reflects the uniform Tur\'an density of a hypergraph $H$,
we arrive at the following question:

\emph{For which palettes $\PP_1,\ldots,\PP_r$ and $\PP'_1,\ldots,\PP'_{r'}$
      does there exist a hypergraph $H$ that is $\PP_i$-colorable for every $i=1,\ldots,r$
      but not $\PP'_j$-colorable for every $j=1,\ldots,r'$?}

The main motivation for this question (as we demonstrate in Section~\ref{sec:example}),
is obtaining a simple tool for constructing hypergraphs with a given uniform Tur\'an density.
Indeed, the case when $r=1$ was resolved independently of us
by King, Piga, Sales and Sch\"ulke\cite{KinPSS} in their work on feasible uniform Tur\'an densities of families of hypergraphs,
where they showed every real that is the Lagrangian of a palette
is the uniform Tur\'an density of a finite family of hypergraphs.

Our main result, which is presented in Theorem~\ref{thm:multi} and Corollary~\ref{cor:multi},
completely answers the question by giving a necessary and sufficient condition,
which is easy to verify for any given palettes $\PP_1,\ldots,\PP_r$ and $\PP'_1,\ldots,\PP'_{r'}$;
note that the condition becomes significantly simpler when $r=1$ (see Theorem~\ref{thm:single}).
We believe that the necessary and sufficient condition that we present has great potential
for identifying hypergraphs with specific uniform Tur\'an densities, and
as evidence to support this claim (in addition to the above mentioned results of King, Piga, Sales and Sch\"ulke),
we apply our results in Section~\ref{sec:example} to show
the existence of a hypergraph with uniform Tur\'an density equal to $4/81$,
which has not been previously known to exist.

We now briefly discuss how the paper is structured.
We introduce the notation needed throughout the paper in Section~\ref{sec:prelim} and
present several Ramsey-type results used in our arguments in Section~\ref{sec:Ramsey}.
In Section~\ref{sec:single}, we present the necessary and sufficient condition in the case $r=1$,
which is significantly simpler, both in terms of the condition and its proof.
Unlike in the general case of $r>1$,
the condition involves only the existence of a naturally defined homomorphism between palettes.
We believe that first presenting the case $r=1$ makes our exposition more accessible.
In Section~\ref{sec:multi}, we present our main results (Theorem~\ref{thm:multi} and Corollary~\ref{cor:multi}).
While the proof follows similar lines as that given in Section~\ref{sec:single},
the condition is significantly more involved (as an additional operation with a palette needs to be introduced,
which we address in more detail in the concluding Section~\ref{sec:concl}) and
the arguments are substantially more complex (as also witnessed by more advanced Ramsey type results employed);
this was also the reason why we decided to first present the arguments in a simpler setting when $r=1$.
In Section~\ref{sec:example}, we present a simple application of our necessary and sufficient condition that
yields the existence of a hypergraph with uniform Tur\'an density $4/81$, and
we conclude in Section~\ref{sec:concl} by discussing the identified necessary and sufficient condition, and
some of its additional future applications.

\section{Preliminaries}
\label{sec:prelim}

In this section,
we introduce notation used throughout the paper and also recall some well-known facts related to the graph regularity method.

We write $[n]$ for the set of the first $n$ positive integers.
If $H$ is a (hyper)graph, then $V(H)$ denotes the vertex set of $H$.
Throughout this paper, all hypergraphs that we consider are $3$-uniform and, as already mentioned in Introduction,
we drop the adjective $3$-uniform for brevity.
An \emph{$\ell$-edge-coloring} of a graph $H$ is an assignment of $\ell$ colors to the edges of $H$,
which needs not be proper, i.e., edges sharing the same vertex may have the same color.
Some of the graphs and hypergraphs in our considerations are equipped with a linear order on their vertex and
we refer to them as \emph{ordered graphs} and \emph{ordered hypergraphs}.
We will carefully say \emph{ordered} (hyper)graph whenever the considered (hyper)graph is ordered,
i.e., the absence of the adjective always means that the considered (hyper)graph is not ordered.
If $H$ is an ordered (hyper)graph, we write $\preceq_H$ for the linear order on the vertex set of $H$;
we drop the subscript in $\preceq_H$ in case the choice of $H$ is obvious form the context.
If $H$ is an ordered (hyper)graph, then $\unorient{H}$ is the (hyper)graph without the linear order, and
if $H$ is a (hyper)graph and $\preceq$ a linear order on $V(H)$,
then $H^{\preceq}$ is the ordered (hyper)graph obtained from $H$ by equipping it with $\preceq$.
Finally, if $H$ is an ordered (hyper)graph,
then $\inv(H)$ is the ordered (hyper)graph with the reverse linear order.

We now introduce some notation related to colorability of (ordered) hypergraphs with palettes,
which is additional to that presented in Introduction.
Consider a palette $\PP=(C,T)$.
If $(x,y,z)\in T$ is a feasible triple,
we refer to the color $x$ as the \emph{left} color, $y$ as the \emph{middle} color and $z$ as the \emph{right} color.
We say that an ordered hypergraph $H$ is \emph{$\PP$-colorable}
if there exists a coloring of the pairs of the vertices of $H$ with colors from $C$ such that
for every edge $uvw$ of $H$ with $u\preceq_H v\preceq_H w$,
there exists a triple $(x,y,z)\in T$ such that
the color of $uv$ is $x$, the color of $uw$ is $y$ and the color $vw$ is $z$.
In particular, a hypergraph $H$ is \emph{$\PP$-colorable}
if there exists a linear order $\preceq$ on $V(H)$ such that $H^{\preceq}$ is $\PP$-colorable.

A \emph{homomorphism} from a palette $\PP=(C,T)$ to a palette $\PP'=(C',T')$
is a mapping $f:C\to C'$ such that $(f(x),f(y),f(z))\in T'$ for every triple $(x,y,z)\in T$.
Observe that if there exists a homomorphism from a palette $\PP$ to a palette $\PP'$,
then every $\PP$-colorable ordered hypergraph is also $\PP'$-colorable.
Indeed, any coloring of pairs of vertices by the colors of $\PP$ that witnesses $\PP$-colorability
gives rise to a coloring of pairs of vertices by the colors of $\PP'$ (simply by their mapping through the homomorphism)
that witnesses $\PP'$-colorability.
We formulate this observation as a proposition for future reference.

\begin{proposition}
\label{prop:hom}
If there exists a homomorphism from a palette $\PP$ to a palette $\PP'$,
then every ordered $\PP$-colorable hypergraph is $\PP'$-colorable.
\end{proposition}

We next define three operations with palettes, starting with the simplest of the three operations.
If $\PP=(C,T)$ is a palette,
then the \emph{inverse} palette, denoted by $\inv(\PP)$, is the palette $(C,T')$ such that
$(x,y,z)\in T'$ iff $(z,y,x)\in T$,
Observe that an ordered hypergraph $H$ is $\PP$-colorable if and only if $\inv(H)$ is $\inv(\PP)$-colorable.

The product of palettes $\PP_1=(C_1,T_1),\ldots,\PP_k=(C_k,T_k)$
is the palette $(C,T)$ such that $C=C_1\times\cdots\times C_k$ and
$((x_1,\ldots,x_k),(y_1,\ldots,y_k),(z_1,\ldots,z_k))\in T$ iff
$(x_i,y_i,z_i)\in T_i$ for every $i\in [k]$.
The product of palettes $\PP_1,\ldots,\PP_k$ is denoted by $\PP_1\times\cdots\times\PP_k$ or simply by $\prod\limits_{i\in [k]}\PP_i$.
Observe that an ordered hypergraph $H$ is $\PP_1$-colorable and $\PP_2$-colorable if and only if
$H$ is $\PP_1\times\PP_2$-colorable.
However,
a hypergraph $H$ that is $\PP_1$-colorable and $\PP_2$-colorable need not be $\PP_1\times\PP_2$-colorable---we give a simple example.
Consider a palette $\PP$ with two colors, $\alpha$ and $\beta$, and
two feasible triples $(\alpha,\beta,\alpha)$ and $(\alpha,\beta,\beta)$.
The hypergraph $K_4^{(3)-}$ is $\PP$-colorable and $\inv(\PP)$-colorable
but it is not $\PP\times\inv(\PP)$-colorable.
Indeed, the only color that appears as the middle in a feasible triple of the palette $\PP\times\inv(\PP)$ is $(\beta,\beta)$.
However, the color $(\beta,\beta)$ is neither the left color nor the right color in any feasible triple.
Consider a vertex order $\preceq$ of $K_4^{(3)-}$, say $v_1\preceq v_2\preceq v_3\preceq v_4$, and
a coloring of pairs of its vertices by the elements $\PP\times\inv(\PP)$.
Note that at most one of the triples $\{v_1,v_2,v_3\}$ and $\{v_1,v_3,v_4\}$ can be an edge of $K_4^{(3)-}$;
which one depends on whether the color of $v_1v_3$ is $(\beta,\beta)$.
Likewise, at most one of the triples $\{v_1,v_2,v_4\}$ and $\{v_2,v_3,v_4\}$ can be an edge.
Hence, the hypergraph $K_4^{(3)-}$ is not $\PP\times\inv(\PP)$-colorable.

\begin{figure}
\begin{center}
\epsfbox{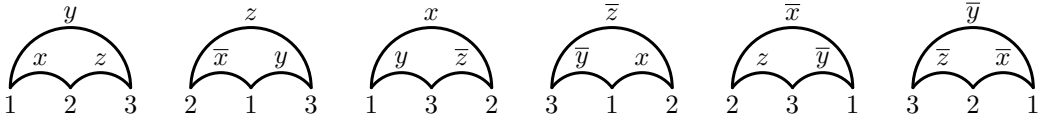}
\end{center}
\caption{Visualization of the triples included in $\sym{T}$ for a triple $(x,y,z)\in T$.}
\label{fig:sym}
\end{figure}

The last of the three operations that we define may look puzzling at the first sight.
Consider a palette $\PP=(C,T)$ such that $C=\{c_1,\ldots,c_k\}$.
The \emph{symmetrization} of the palette $\PP$, denoted by $\sym{\PP}$,
is the palette $(\sym{C},\sym{T})$ such that
\begin{align*}
\sym{C}  =&\{ c_1,\ldots,c_k,\symc{c_1},\ldots,\symc{c_k}\} \mbox{ and } \\
\sym{T}  =&\bigcup_{(x,y,z)\in T}\{(x,y,z),(\symc{x},z,y),(y,x,\symc{z}),(\symc{y},\symc{z},x), (z,\symc{x},\symc{y}),(\symc{z},\symc{y},\symc{x})\},
\end{align*}
where $\symc{c_1},\dots,\symc{c_k}$ are new colors. Informally speaking,
the palette $\sym{\PP}$ has two twin colors for each color of the palette $\PP$, the original color and its ``clone'', and
it contains six triples for each triple $(x,y,z)$ contained in $\PP$ that are obtained as follows (see Figure~\ref{fig:sym}):
orient the edges of an ordered triangle left to right and color the edge between the two left vertices $x$, the edge between the two extreme vertices $y$ and the edge between the two right vertices $z$ so that the colored triangle corresponds to the triple $(x,y,z)$.
Now reorder the vertices arbitrarily,
recolor all edges where the direction of the edge does not agree with the order to the clone of its original color, and
add the resulting triple of colors to the family.
The six possible orders of the vertices yield the six triples in $\sym{\PP}$ resulting from each $(x,y,z)\in T$.

We next recall a version of Szemer\'edi Regularity Lemma for prepartitioned edge-colored graphs, and
present two related lemmas that we will use later.
To do so, we need to introduce some additional terminology.
We write $Z=X_1\dot\cup\cdots\dot\cup X_k$ for a \emph{partition} of a set $Z$,
i.e., a decomposition of $Z$ into disjoint subsets.
A partition is an \emph{equipartition} if the sizes of the subsets differ by at most one,
a partition $Y_1\dot\cup\cdots\dot\cup Y_{\ell}$ is a \emph{refinement} of a partition $X_1\dot\cup\cdots\dot\cup X_k$
if for every $i\in [\ell]$, there exists $j\in [k]$ such that $Y_i\subseteq X_j$, and
a partition $Y_1\dot\cup\cdots\dot\cup Y_{\ell}$ is an \emph{equirefinement} of a partition $X_1\dot\cup\cdots\dot\cup X_k$
if $Y_1\dot\cup\cdots\dot\cup Y_{\ell}$ is a refinement of $X_1\dot\cup\cdots\dot\cup X_k$ and
each part $X_i$, $i\in [k]$, is split into exactly $\ell/k$ parts.

We now state a version of Szemer\'edi Regularity Lemma for prepartitioned edge-colored graphs.

\begin{lemma}
\label{lm:SRL}
For every $\varepsilon>0$, $m\in\NN$, $\kappa\in\NN$ and $k_0\in\NN$,
there exists $K_0\in\NN$ such that
for every $m$-edge-colored complete graph $G$ and
every equipartition $U_1\dot\cup\cdots\dot\cup U_{\kappa}$ of $V(G)$,
there exists an equipartition $V_1\dot\cup\cdots\cup V_k$ that is an equirefinement of $U_1\dot\cup\cdots\dot\cup U_{\kappa}$ such that $k_0\le k\le K_0$ and
all but $\varepsilon\binom{k}{2}$ pairs $(V_i,V_j)$, $1\le i<j\le k$, are $\varepsilon$-regular,
i.e., they satisfy for every color $x$, every subset $A\subseteq V_i$, $|A|\ge\varepsilon |V_i|$, and every subset $B\subseteq V_j$, $|B|\ge\varepsilon |V_j|$ that
\[\left|\frac{e_x(V_i,V_j)}{|V_i|\;|V_j|}-\frac{e_x(A,B)}{|A|\;|B|}\right|\le\varepsilon,\]
where $e_x(X,Y)$ is the number of edges colored with $x$ between subsets $X$ and $Y$ of $V(G)$.
\end{lemma}

The proof of the next lemma follows from standard regularity method arguments, and we omit it.

\begin{lemma}
\label{lm:triangle}
Let $G$ be a graph and $\varepsilon\in (0,1/2)$.
Further let $V_1$, $V_2$ and $V_3$ be disjoint subsets of $V(G)$ such that each pair of them is $\varepsilon$-regular,
i.e., they satisfy for every $1\le i<j\le3$, for every subset $A\subseteq V_i$, $|A|\ge\varepsilon |V_i|$, and every subset $B\subseteq V_j$, $|B|\ge\varepsilon |V_j|$, that
\[\left|\frac{e(V_i,V_j)}{|V_i|\;|V_j|}-\frac{e(A,B)}{|A|\;|B|}\right|\le\varepsilon,\]
where $e(X,Y)$ is the number of edges between subsets $X$ and $Y$ of $V(G)$.
If $\frac{e(V_i,V_j)}{|V_i|\;|V_j|}\ge 2\varepsilon$ for all $1\le i<j\le 3$,
then there exist mutually adjacent vertices $v_1\in V_1$, $v_2\in V_2$ and $v_3\in V_3$,
i.e., $v_1v_2v_3$ is a triangle in $G$.
\end{lemma}

We finish this section with stating an auxiliary lemma,
which permits choosing representative parts with respect to a given initial equipartition
such that all the pairs of the representative parts are $\varepsilon$-regular.
While the lemma follows a standard argument used in conjunction with the regularity method,
we include a short proof for completeness.

\begin{lemma}
\label{lm:clique}
Let $\kappa\in\NN$ and let $G$ be an edge-colored complete graph with an equipartition $U_1\dot\cup\cdots\dot\cup U_{\kappa}$ of $V(G)$.
If an equipartition $V_1\dot\cup\cdots\cup V_k$ is a equirefinement of $U_1\dot\cup\cdots\dot\cup U_{\kappa}$ such that
all but $\varepsilon\binom{k}{2}$ pairs $(V_i,V_j)$, $1\le i<j\le k$, are $\varepsilon$-regular for some $\varepsilon\le\frac{2}{\kappa^2}$,
then there exist $j_1,\ldots,j_{\kappa}\in [k]$ such that $V_{j_i}\subseteq U_i$ for every $i\in [\kappa]$, and
every pair $(V_{j_i},V_{j_{i'}})$, $1\le i<i'\le\kappa$, is $\varepsilon$-regular.
\end{lemma}

\begin{proof}
Pick $j_i$ uniformly at random from among the $k/\kappa$ indices with $V_{j_i}\subseteq U_i$ independently for all $i\in[\kappa]$.
Any pair of distinct parts $\{V_j,V_{j'}\}$ appears among the $k$ sets selected with probability $(\kappa/k)^2$, and
so the union bound yields that
the probability that one of the at most $\varepsilon\binom k2$ non-$\varepsilon$-regular pairs
appears among the pairs formed by the selected parts $V_{j_i}$, $i\in [\kappa]$,
is at most $\varepsilon\binom{k}{2}(\kappa/k)^2<\varepsilon\kappa^2/2\le1$.
We conclude that that all pairs formed by the parts $V_{j_i}$, $i\in [\kappa]$,
are $\varepsilon$-regular with positive probability.
\end{proof}

\section{Ramsey type results}
\label{sec:Ramsey}

In this section, we present results from Ramsey theory that we use in our arguments.
We start with two classical results,
the Erd\H os-Szekeres Theorem and a version of Ramsey's Theorem.

\begin{theorem}[{Erd\H os-Szekeres Theorem~\cite{ErdS35}}]
\label{thm:ES}
Let $k$ and $\ell$ be two integers.
Any sequence of $(k-1)(\ell-1)+1$ distinct reals
contains an increasing subsequence of length $k$ or a decreasing subsequence of length $\ell$.
\end{theorem}

\begin{theorem}[{Ramsey's Theorem~\cite{Ram30}, multicolor hypergraph version}]
\label{thm:Ramsey}
For every $n\in\NN$, $\ell\in\NN$, and $r\in\NN$
there exists $R\in\NN$ such that
for any $\ell$-coloring of the $r$-element subsets of $[R]$ there exists a monochromatic subset of size $n$, that is, $S\subseteq[R]$ with $|S|=n$ and all $r$-element subsets of $S$ having the same color.
\end{theorem}

We believe that the next theorem may be known but we have not been able to identify a proper reference.
We refer to Figure~\ref{fig:grid} for the illustration of the statement.
Recall that $\sgn$ is the function such that
$\sgn(x)=+1$ if $x>0$, $\sgn(x)=0$ for $x=0$, and $\sgn(x)=-1$, otherwise.

\begin{figure}
\begin{center}
\epsfbox{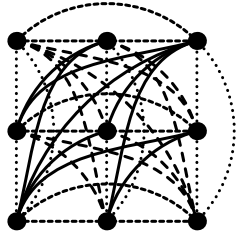}
\end{center}
\caption{Visualization of a sought ``subgrid'' in the statement of Theorem~\ref{thm:grid} for $n=3$ and $r=2$.
         The colors of pairs of points are depicted by (four) different types of lines.}
\label{fig:grid}
\end{figure}

\begin{theorem}
\label{thm:grid}
For every $n\in\NN$, $r\in\NN$ and $\ell\in\NN$,
there exists $R\in\NN$ such that the following holds.
For any function $F:[R]^r\times [R]^r\to [\ell]$,
there exist $n$-element sets $J_1,\ldots,J_r\subseteq [R]$ and
a function $C:\{-1,0,1\}^r\to [\ell]$ such that
\[F(x,y)=C\left(\sgn(y_1-x_1),\ldots,\sgn(y_r-x_r)\right).\]
for any two $r$-tuples $x=(x_1,\ldots,x_r)$ and $y=(y_1,\ldots,y_r)$ contained in $J_1\times\cdots\times J_r$
\end{theorem}

\begin{proof}
Apply Ramsey's Theorem (Theorem~\ref{thm:Ramsey})
with $r(n+1)$ in place of $n$, $2r$ in place of $r$ and $\ell^{3^r}$ in place of $\ell$ to get $R$.
Consider a function $F:[R]^r\times [R]^r\to [\ell]$.
We now color $(2r)$-element subsets of $[R]$ with $\ell^{3^r}$ colors;
we identify the $\ell^{3^r}$ colors with functions from $\{-1,0,1\}^r$ to $[\ell]$, and
a $(2r)$-element subset $\{x_1,\ldots,x_{2r}\}$, $x_1<\cdots<x_{2r}$,
is colored with the following function $C:\{-1,0,1\}^r\to [\ell]$:
\[C(z_1,\ldots,z_r)=F\left((x_{i_1},x_{i_2},\ldots,x_{i_r}),(x_{j_1},x_{j_2},\ldots,x_{{j_r}})\right)\]
where $(i_k,j_k)=(2k-1,2k)$ if $z_k=+1$, $(i_k,j_k)=(2k-1,2k-1)$ if $z_k=0$, and $(i_k,j_k)=(2k,2k-1)$ otherwise.
By the choice of $R$,
there exists an $(r(n+1))$-element subset $S\subseteq [R]$ such that
all $(2r)$-element subsets of $S$ have the same color $C$;
this $C$ will be the function from the statement of the theorem.
Let $x_1,\ldots,x_{r(n+1)}$ be the elements of $S$ listed in the increasing order.
Set $J_1$ to be the set containing the elements $x_1,\ldots,x_n$,
$J_2$ the set containing the elements $x_{n+2},\ldots,x_{2n+1}$,
$J_3$ the set containing the elements $x_{2n+3},\ldots,x_{3n+2}$, and so forth (one element always needs to be skipped
because all $r$-tuples with equal coordinates will not be constrained otherwise).
It is easy to verify that the sets $J_1,\ldots,J_r$ have the properties given in the statement of the theorem.
\end{proof}

We conclude this section with the following theorem, which was proven by Fishburn and Graham~\cite{FisG93};
we remark that the best known quantitative bound can be found in~\cite{BucST23}.
Informally speaking,
the theorem asserts that for every linear order $\preceq$ of a sufficiently large grid $[R]^r$,
there exists a subgrid such that
the linear order $\preceq$ is the lexicographic order of the $r$-tuples
after a suitable permutation of the coordinates and
possibly reversing the linear orders in some of the coordinates.

\begin{theorem}
\label{thm:ES-grid}
For every $n\in\NN$ and $r\in\NN$,
there exists $R\in\NN$ such that the following holds for any linear order $\preceq$ on $[R]^r$.
There exist $n$-element sets $J_1,\ldots,J_r\subseteq [R]$,
a permutation $\pi$ of order $r$, and a function $s:[r]\to\{-1,+1\}$ such that
the linear order $\preceq$ of the $r$-tuples $(x_1,\ldots,x_r)\in J_1\times\cdots\times J_r$
is the lexicographic order given by $(s_1x_{\pi(1)},\ldots,s_rx_{\pi(r)})$.
\end{theorem}

\section{Single palette}
\label{sec:single}

In this section, we discuss the case of a single palette.
While the results presented in this section are particular cases of those presented in Section~\ref{sec:multi},
we decided to include them to demonstrate the use of arguments employed in Section~\ref{sec:multi} in a simpler setting.
We believe that this makes our presentation more accessible.
Vaguely speaking and simplifying the actual state,
the argument using Theorems~\ref{thm:grid} and~\ref{thm:ES-grid} presented in Section~\ref{sec:multi}
simplifies to an argument using Theorems~\ref{thm:Ramsey} and~\ref{thm:ES} presented in this section, respectively.

We start with a lemma that relates the existence of an ordered hypergraph colorable by one palette and not by another
to the existence of a homomorphism between the palettes.
Note that the implication in the lemma is actually an equivalence, since the reverse implication is stated as Proposition~\ref{prop:hom}.

\begin{lemma}
\label{lm:exist}
Let $\PP$ and $\PP_0$ be two palettes.
If every ordered hypergraph that is $\PP$-colorable is also $\PP_0$-colorable,
then there exists a homomorphism from the palette $\PP$ to the palette $\PP_0$.
\end{lemma}

\begin{proof}
Fix palettes $\PP=(C,T)$ and $\PP_0=(C_0,T_0)$ for the proof and
suppose that every ordered $\PP$-colorable hypergraph is $\PP_0$-colorable.
Let $R$ be an integer such that any coloring of the edges of $K_R$ with $|C_0|^{|C|}$ colors contains a monochromatic triangle;
such $R$ exists by Ramsey's Theorem (Theorem~\ref{thm:Ramsey} applied with $r=2$, $n=3$ and $\ell=|C_0|^{|C|}$).
We apply Lemma~\ref{lm:SRL} with
\[\kappa=k_0=R\mbox{,}\qquad m=|C|\cdot |C_0|\qquad\mbox{and}\qquad\varepsilon=\min\left\{\frac{1}{100m},\frac{2}{R^2}\right\}\]
to get $K_0$.
Next choose $N\ge 2K_0$ to be a sufficiently large integer divisible by $R$ so that it holds that
\[
3^N\cdot |C|\cdot e^{-\frac{\varepsilon^2 N^2}{8|C|\;K_0^2}}<1.
\]
Consider an ordered complete graph $G$ with $N$ vertices and color its edges uniformly at random with the colors from $C$.
We observe using Chernoff Bound that if $A$ and $B$ are two disjoint subsets of $V(G)$, each of size at least $\lfloor N/K_0\rfloor$ (note that $\lfloor N/K_0\rfloor\ge N/2K_0$ as $N\ge 2K_0$) and $x\in C$,
then the probability that the number of edges between $A$ and $B$ colored with $x$ is smaller than $(1-\varepsilon)\frac{|A|\;|B|}{|C|}$
is at most
\[e^{-\frac{\varepsilon^2|A|\;|B|}{2|C|}}\le e^{-\frac{\varepsilon^2N^2}{8|C|\;K_0^2}}.\]
Since the number of pairs of disjoints sets $A$ and $B$ with size at least $\lfloor N/K_0\rfloor$ is at most $3^N$,
we conclude using the choice of $N$ that the probability that
there exists a pair of disjoints sets $A$ and $B$, each of size at least $\lfloor N/K_0\rfloor$, such that
the number of edges between $A$ and $B$ colored with any single color $x\in C$ is smaller than $(1-\varepsilon)\frac{|A|\;|B|}{|C|}$
is less than $1$.
Hence, there exists a coloring of the edges of $G$ such that
the number of edges colored with each color $x\in C$
is at least $(1-\varepsilon)\frac{|A|\;|B|}{|C|}$ for any two such disjoint sets $A$ and $B$;
fix such a coloring for the rest of the proof.

Let $H$ be the ordered hypergraph with $V(H)=V(G)$ such that
three vertices $v\preceq v'\preceq v''$ form an edge iff $(x,y,z)\in T$ where $x$ is the color of $vv'$, $y$ is the color of $vv''$ and $z$ is the color of $v'v''$.
The construction of the ordered hypergraph $H$ ensures that it is $\PP$-colorable.
Therefore, the ordered hypergraph $H$ is also $\PP_0$-colorable;
fix a coloring of pairs of the vertices of $H$ witnessing this and view it as a coloring of the edges of $G$.
Hence, each edge of $G$ is now colored with a pair $(x,x')\in C\times C_0$.
Observe that it holds for any three vertices $v\preceq v'\preceq v''$ that
if $(x,x')$ is the color of $vv'$, $(y,y')$ the color of $vv''$ and $(z,z')$ the color of $v'v''$, and $(x,y,z)\in T$,
then $(x',y',z')\in T_0$.

Let $U_1\dot\cup\cdots\dot\cup U_R$ be the equipartition of $V(G)$ such that $v\preceq v'$ for all $v\in U_i$ and $v'\in U_{i'}$ with $1\le i<i'\le R$.
Next apply Lemma~\ref{lm:SRL} to the graph $\unorient{G}$ with the edge-coloring given by the pairs of colors from $C\times C_0$ and
the equipartition $U_1\dot\cup\cdots\dot\cup U_R$ to get an equipartition $V_1\dot\cup\cdots\dot\cup V_k$, $k_0\le k\le K_0$, with the properties given in the statement of the lemma.
Lemma~\ref{lm:clique} implies that there exist indices $j_1,\ldots,j_R$ such that $V_{j_i}\subseteq U_i$ for all $i\in [R]$ and
each pair $(V_{j_i},V_{j_{i'}})$, $1\le i<i'\le R$, is $\varepsilon$-regular.
Fix such indices $j_1,\ldots,j_R$ for the rest of the proof.

We now form an auxiliary edge-colored complete graph $G'$ with $R$ vertices;
the $R$ vertices of $G'$ correspond to the parts $V_{j_1},\ldots,V_{j_R}$.
For each pair $i$ and $i'$ such that $1\le i<i'\le R$,
fix a mapping $f:C\to C'$ such that
$f(x)$, $x\in C$, is the most frequent color $x'\in C_0$ of the edges between $V_{j_i}$ and $V_{j_{i'}}$ that are colored with $x\in C$ (ties are resolved arbitrarily).
Recall that for every $x\in C$, at least $(1-\varepsilon)\frac{|V_{j_i}|\;|V_{j_{i'}}|}{|C|}$ edges between $V_{j_i}$ and $V_{j_{i'}}$ are colored with $x$, and
thus there are at least $(1-\varepsilon)\frac{|V_{j_i}|\;|V_{j_{i'}}|}{m}$ edges between $V_{j_i}$ and $V_{j_{i'}}$ colored with the pair $(x,f(x))\in C\times C_0$.
We view the graph $G'$ as edge-colored with the mappings $f$ and deduce using Ramsey's Theorem that $G'$ contains a monochromatic triangle,
i.e., there exist indices $i_1<i_2<i_3$ such that
each pair of the parts $V_{j_{i_1}}\subseteq U_{i_1}$, $V_{j_{i_2}}\subseteq U_{i_2}$ and $V_{j_{i_3}}\subseteq U_{i_3}$
is associated with the same mapping from $C$ to $C_0$;
let $f_0$ be this mapping.
Since each pair of the parts $V_{j_{i_1}}$, $V_{j_{i_2}}$ and $V_{j_{i_3}}$ is $\varepsilon$-regular,
Lemma~\ref{lm:triangle} implies that for all $x_{12},x_{13},x_{23}\in C$,
there exist vertices $v_1\in V_{j_{i_1}}$, $v_2\in V_{j_{i_2}}$ and $v_3\in V_{j_{i_3}}$ such that
the edge $v_1v_2$ in the graph $G$ is colored with $(x_{12},f_0(x_{12}))$,
the edge $v_1v_3$ with $(x_{13},f_0(x_{13}))$, and
the edge $v_2v_3$ with $(x_{23},f_0(x_{23}))$.
In particular, if $(x_{12},x_{13},x_{23})\in T$,
then the ordered hypergraph $H$ contains the edge $v_1v_2v_3$ and so $(f_0(x_{12}),f_0(x_{13}),f_0(x_{23}))\in T_0$.
We conclude that the mapping $f_0$ is a homomorphism from the palette $\PP$ to the palette $\PP_0$.
\end{proof}

We now prove another lemma that is needed to prove the main result of this section.
The lemma relates the existence of a hypergraph colorable by one palette but not by another
to the existence of ordered hypergraphs colorable by the former palette but not by the latter.

\begin{lemma}
\label{lm:single}
Let $\PP$ and $\PP_0$ be two palettes.
If there exists an ordered hypergraph $H^+$ that is $\PP$-colorable but not $\PP_0$-colorable and
there exists an ordered hypergraph $H^-$ that is $\PP$-colorable but not $\inv(\PP_0)$-colorable,
then there exists a hypergraph $H$ that is $\PP$-colorable but not $\PP_0$-colorable.
\end{lemma}

\begin{proof}
Fix the palettes $\PP=(C,T)$ and $\PP_0$, and the ordered hypergraphs $H^+$ and $H^-$ as in the statement of the lemma.
Let $n^+$ and $n^-$ be the numbers of vertices of the hypergraphs $H^+$ and $H^-$, respectively, and
set $n=\max\{n^+,n^-\}$ and $k=(n^+-1)(n^--1)+1$.
Finally, let $N$ be a positive integer such that
\[N!\cdot(1-|C|^{-n^2})^{\frac{N^2}{k^4}}<1;\]
such $N$ exists since $N!$ is at most $e^{N\log N}$ and the integers $k$, $n$ and $|C|$ are fixed.
The vertex set $V(H)$ of the sought hypergraph $H$ will be $[N]$,
in particular, the vertex set $V(H)$ is equipped with a natural linear order (although $H$ is not an ordered hypergraph).
Color the pairs of vertices of $H$ with colors from the set $C$ uniformly at random and
include as an edge any triple that is colored consistently with the palette $\PP$ and the linear order on $V(H)$.
Clearly, the constructed hypergraph $H$ is $\PP$-colorable.

We now show that the hypergraph $H$ is not $\PP_0$-colorable with positive probability.
Let $F_1,\ldots,F_m$ be any inclusion-wise maximal family of $k$-element subsets of $V(H)$ such that
no two sets $F_i$ and $F_j$, $1\le i<j\le m$, have more than a single vertex in common.
Since a single $k$-vertex set blocks less than $\binom{k}{2}\binom{N-2}{k-2}$ other sets from the inclusion in the family,
the number $m$ of the sets is at least
\begin{equation}
m\ge\frac{\binom{N}{k}}{\binom{k}{2}\binom{N-2}{k-2}}=\frac{2N(N-1)}{k^2(k-1)^2}\ge\frac{N^2}{k^4}.
\label{eq:Fnum}
\end{equation}
Fix any linear order $\preceq$ on $V(H)$.
Consider a set $F_i$, $i\in [m]$.
By the Erd\H os-Szekeres Theorem (Theorem~\ref{thm:ES}),
the set $F_i$ contains an $n^+$ elements that form an increasing sequence both in the natural order on $V(H)=[N]$ and in $\preceq$ or
$n^-$ elements that form a decreasing sequence in the natural order on $V(H)=[N]$ but an increasing sequence in $\preceq$.
In the former case,
the probability that the ordered hypergraph $H^{\preceq}$ restricted to the $n^+$ elements contains a copy of $H^+$
is at least $|C|^{-\binom{n^+}{2}}$ and
so the restriction of the ordered hypergraph $H^{\preceq}$ to $F_i$ is not $\PP_0$-colorable
with probability at least $|C|^{-\binom{n^+}{2}}\ge |C|^{-n^2}$.
In the latter case,
the probability that the ordered hypergraph $H^{\preceq}$ restricted to the $n^-$ elements contains a copy of $\inv(H^-)$
is at least $|C|^{-\binom{n^-}{2}}$ and
so the restriction of the ordered hypergraph $H^{\preceq}$ to $F_i$ is not $\PP_0$-colorable
with probability at least $|C|^{-\binom{n^-}{2}}\ge |C|^{-n^2}$.
We conclude that the probability that
the restriction of the ordered hypergraph $H^{\preceq}$ to $F_i$ is $\PP_0$-colorable is at most $1-|C|^{-n^2}$.

Since no two sets $F_i$ and $F_j$, $1\le i<j\le m$, have more than a single vertex in common,
the events that
the restriction of the ordered hypergraph $H^{\preceq}$ to $F_i$ is $\PP_0$-colorable
are independent for all $i\in [m]$.
Hence, using \eqref{eq:Fnum}, we obtain that the probability that
the ordered hypergraph $H^{\preceq}$ is $\PP_0$-colorable is at most
\[(1-|C|^{-n^2})^m\le (1-|C|^{-n^2})^{\frac{N^2}{k^4}}.\]
Since the hypergraph $H$ is $\PP_0$-colorable if and only if
there exists a linear order $\preceq$ such that the ordered hypergraph $H^{\preceq}$ is $\PP_0$-colorable and
the number of linear orders of an $N$-element set is $N!$,
the probability that the hypergraph $H$ is $\PP_0$-colorable is at most
\[N!\cdot (1-|C|^{-n^2})^{\frac{N^2}{k^4}},\]
which is less than one by the choice of $N$.
We conclude that the hypergraph $H$ is not $\PP_0$-colorable with positive probability and
so the hypergraph $H$ has the properties given in the statement of the lemma with positive probability,
which in particular establishes the existence of such a hypergraph.
\end{proof}

We are now ready to prove the main result of this section.

\begin{theorem}
\label{thm:single}
Let $\PP$ and $\PP_0$ be two palettes.
There exists a hypergraph that is $\PP$-colorable but not $\PP_0$-colorable if and only if
there is no homomorphism from the palette $\PP$ to the palette $\PP_0$ or to the palette $\inv(\PP_0)$.
\end{theorem}

\begin{proof}
Fix the palettes $\PP$ and $\PP_0$.
We prove the equivalence of the statement by proving each implication separately,
starting with showing the contrapositive of the ``only if'' part.
This means that we will show that
if there is a homomorphism from the palette $\PP$ to the palette $\PP_0$ or to the palette $\inv(\PP_0)$,
then there is no hypergraph that is $\PP$-colorable but not $\PP_0$-colorable,
i.e., every $\PP$-colorable hypergraph is also $\PP_0$-colorable.
If there exists a homomorphism from the palette $\PP$ to the palette $\PP_0$,
then every $\PP$-colorable hypergraph $H$ is $\PP_0$-colorable by Proposition~\ref{prop:hom} (consider $H$ as ordered hypergraph
with the order witnessing its $\PP$-colorability and apply Proposition~\ref{prop:hom}).
Similarly,
if there exists a homomorphism from the palette $\PP$ to the palette $\inv(\PP_0)$,
then every $\PP$-colorable hypergraph $H$ is $\inv(\PP_0)$-colorable and
so it is also $\PP_0$-colorable (keep the colors of the pairs of vertices and reverse the vertex order).

It remains to show the ``if'' part of the equivalence.
Assume that there is no homomorphism from the palette $\PP$ to the palette $\PP_0$ and
there is no homomorphism from the palette $\PP$ to the palette $\inv(\PP_0)$.
Lemma~\ref{lm:exist} yields that
there exist an ordered hypergraph $H^+$ that is $\PP$-colorable but not $\PP_0$-colorable and
an ordered hypergraph $H^-$ that is $\PP$-colorable but not $\inv(\PP_0)$-colorable.
The existence of a hypergraph that is $\PP$-colorable but not $\PP_0$-colorable
now follows from Lemma~\ref{lm:single}.
\end{proof}

\section{Multiple palettes}
\label{sec:multi}

In this section, we prove our main result.
The structure of the proof is similar to the proof of Theorem~\ref{thm:single},
however, there are various challenges to be overcome,
which required additional ideas and more refined arguments (as also evidenced
by more involved Ramsey theory results employed in this section).

We start with the version of Lemma~\ref{lm:exist} for multiple palettes.
As in the case of Lemma~\ref{lm:exist},
the implication in this lemma is actually an equivalence.
Recall that $\unorient{H}$ stands for the unordered hypergraph underlying the ordered hypergraph $H$.

\begin{lemma}
\label{lm:existm}
Let $\PP_1,\ldots,\PP_r$ and $\PP_0$ be $r+1$ palettes.
If every ordered hypergraph $H$ such that $H$ is $\PP_1$-colorable and
$\unorient{H}$ is $\PP_s$-colorable for every $s\in [r]\setminus\{1\}$
is also $\PP_0$-colorable,
then there exists a homomorphism from the palette $\PP_1\times\prod\limits_{s=2}^r\sym{\PP_s}$ to the palette $\PP_0$.
\end{lemma}

\begin{proof}
Fix palettes $\PP_1=(C_1,T_1),\ldots,\PP_r=(C_r,T_r)$ and $\PP_0=(C_0,T_0)$ for the proof, and
suppose that every ordered hypergraph $H$ such that $H$ is $\PP_1$-colorable and
$\unorient{H}$ is $\PP_s$-colorable for every $s\in [r]\setminus\{1\}$
is $\PP_0$-colorable.
Let $\ell_C$ be the product of the sizes of the sets $C_s$, $s\in [r]$, and
let $R$ be an integer obtained by applying Theorem~\ref{thm:grid} with $n=3$, $r$ and $\ell=|C_0|^{\ell_C}$.
We apply Lemma~\ref{lm:SRL} with
\[\kappa=k_0=R^r\mbox{, }m=|C_0|\ell_C\mbox{ and }\varepsilon=\min\left\{\frac{1}{100m},\frac{2}{\kappa^2}\right\}\]
to get $K_0$.
Next choose $N\ge 2K_0$ to be a sufficiently large integer divisible by $R^r$ so that it holds that
\[
3^N\cdot \ell_C\cdot e^{-\frac{\varepsilon^2 N^2}{8\ell_C\;K_0^2}}<1.
\]
We next consider an $N$-vertex complete graph $G$ 
with an equipartition of its vertex set to the sets $U_i$ for $i\in[R]^r$, and
define $r$ linear orders $\preceq_1,\ldots,\preceq_r$ on $V(G)$ such that the following holds:
if $v\in U_{i_1,\ldots,i_r}$, $v'\in U_{i'_1,\ldots,i'_r}$ and $i_q<i'_q$, $q\in [r]$, then $v\preceq_q v'$.
Note that this determines the $r$ orders between some of the parts of the equipartition and $\preceq_1,\ldots,\preceq_r$ are arbitrary linear extensions of the predefined partial orders.
We next color the edges of $G$ uniformly at random with colors from $C_1\times\cdots\times C_r$, and
observe using Chernoff Bound that if $A$ and $B$ are two disjoint subsets of $V(G)$,
each of size at least $\lfloor N/K_0\rfloor$ and
$x\in C_1\times\cdots\times C_r$,
then the probability that the number of edges between $A$ and $B$ colored with $x$ is smaller than $(1-\varepsilon)\frac{|A|\;|B|}{\ell_C}$
is at most \[e^{-\frac{\varepsilon^2|A|\,|B|}{2\ell_C}}\le e^{-\frac{\varepsilon^2N^2}{8\ell_C\;K_0^2}}.\]
Since the number of pairs of disjoints sets $A$ and $B$ with size at least $\lfloor N/K_0\rfloor$ is at most $3^N$,
we conclude using the choice of $N$ that the probability that
there exists a pair of disjoints sets $A$ and $B$, each of size at least $\lfloor N/K_0\rfloor$, such that
the number of edges between $A$ and $B$ colored with any single color $x\in C_1\times\cdots\times C_r$
is smaller than $(1-\varepsilon)\frac{|A|\;|B|}{\ell_C}$
is less than $1$.
Hence, there exists a coloring of the edges of $G$ such that
the number of edges colored with each color $x\in C_1\times\cdots\times C_r$
is at least $(1-\varepsilon)\frac{|A|\;|B|}{\ell_C}$ for any two such disjoint sets $A$ and $B$;
fix such a coloring for the rest of the proof.

Let $H$ be the ordered hypergraph with $V(H)=V(G)$ equipped with the linear order $\preceq_1$ such that
three vertices $v_1$, $v_2$ and $v_3$ form an edge in $H$ iff
for every $q\in [r]$, the edge-coloring of $G^{\preceq_q}$
restricted to the $q$-th coordinates of the colors
is consistent with the palette $\PP_q$ on the vertices $v_1$, $v_2$ and $v_3$,
i.e., $(x,y,z)\in T_q$ where
$x$ is the $q$-th coordinate of the color of the edge $v_{\alpha}v_\beta$,
$y$ is the $q$-th coordinate of the color of the edge $v_{\alpha}v_\gamma$,
$z$ is the $q$-th coordinate of the color of the edge $v_{\beta}v_\gamma$, and
the indices $\alpha,\beta,\gamma\in [3]$ are chosen so that $v_\alpha\prec_q v_\beta\prec_q v_\gamma$.
In particular,
the ordered hypergraph $H$ is $\PP_1$-colorable and
the hypergraph $\unorient{H}$ is $\PP_q$-colorable for every $q\in [r]\setminus\{1\}$.
It follows that the ordered hypergraph $H$ is also $\PP_0$-colorable.
Fix a coloring of the pairs of vertices of $H$ witnessing this and view it as the second edge-coloring of $G$ (the first
edge-coloring is the edge-coloring with the elements from $C_1\times\cdots\times C_r$ fixed in the previous paragraph).

We next apply Lemma~\ref{lm:SRL} to the graph $G$ with the edge-coloring given by the pairs of the colors
from the first and the second edge-colorings constructed earlier,
i.e., every edge of $G$ is colored with an element of $C_1\times\cdots\times C_r\times C_0$, and
with the equipartition of its vertex set to the sets $U_i$ for $i\in[R]^r$.
Lemma~\ref{lm:SRL} yields
an equipartition $V_1\dot\cup\cdots\dot\cup V_k$, $k\le K_0$,
with the properties given in its statement,
in particular, this equipartition
is an equirefinement of $U_{1, \ldots, 1}\dot\cup\cdots\dot\cup U_{R, \ldots, R}$.
Lemma~\ref{lm:clique} yields that
there exist indices $j_i\in [k]$ for $i\in[R]^r$ such that
$V_{j_i}\subseteq U_i$ for every $i\in [R]^r$ and for $i,i'\in[R]^r$, $i\ne i'$ the pair $(V_{j_i},V_{j_{i'}})$ is $\varepsilon$-regular.
Fix such $R^r$ indices.

We now aim at applying Theorem~\ref{thm:grid};
recall that we used the parameters $n=3$, $r$ and $\ell=|C_0|^{\ell_C}$ to get the value of $R$.
Let $\cal F$ denote the set of all functions from $C_1\times\cdots\times C_r$ to $C_0$;
for the purpose of applying Theorem~\ref{thm:grid},
we identify the functions in $\cal F$ with the colors $1,\ldots,\ell$.
For $i,i'\in [R]^r$, $i\ne i'$,
we define $F(i,i')$ to be the function $f\in\cal F$ that
such that $f(x)$ is the most frequent color in the second coloring of the edges between the sets $V_{j_i}$ and $V_{j_{i'}}$ that
are colored with $x$ in the first coloring (ties resolved arbitrarily).
Theorem~\ref{thm:grid} yields that
there exist $3$-element subsets $J_1,\ldots,J_r\subseteq [R]$ and
a function $\Gamma:\{-1,0,1\}^r\to\cal F$ such that
\[F((i_1, \ldots, i_r),(i'_1, \ldots, i'_r))=\Gamma(\sgn(i_1-i'_1),\dots,\sgn(i_r-i'_r))\]
for all distinct $(i_1, \ldots, i_r),(i'_1, \ldots, i'_r)\in J_1\times\cdots\times J_r$.

We are now ready to define a function $h:C_1\times\prod\limits_{s=2}^r\sym{C_s}\to C_0$,
which we will next show to be a homomorphism from $\PP_1\times\prod\limits_{s=2}^r\sym{\PP_s}$ to $\PP_0$.
We will use the following notation:
if $x$ is a color of a palette $\sym{C_i}$, $i\in [s]\setminus\{1\}$,
then $\hat x$ is the corresponding color of the palette $C_i$,
i.e., the color $x$ is either $\hat x$ or the clone $\symc{\hat x}$ of $\hat x$.
To simplify the notation, we also set $\hat x$ to be $x$ itself for each color $x$ of $C_1$.
Consider $x=(x_1,\ldots,x_r)\in C_1\times\prod\limits_{s=2}^r\sym{C_s}$, and
define $\gamma(x)\in\{-1,+1\}^r$ as
\[\gamma(x)_i=\begin{cases}
              +1 & \mbox{if $x_i=\hat x_i$, and} \\
	      -1 & \mbox{otherwise.}
              \end{cases}\]
Note that $\gamma_1(x)=1$ for every $x\in C_1\times\prod\limits_{s=2}^r\sym{C_s}$.
We set \[h(x)=\Gamma(\gamma(x))(\widehat{x_1},\ldots,\widehat{x_r})\]
for every $x\in C_1\times\prod\limits_{s=2}^r\sym{C_s}$;
note that $\Gamma(\gamma(x))\in\cal F$,
i.e., $\Gamma(\gamma(x))$ is a function from $C_1\times\cdots\times C_r$ to $C_0$.

We need to verify that $h$ is a homomorphism from $\PP_1\times\prod\limits_{s=2}^r\sym{\PP_s}$ to $\PP_0$.
Let $x=(x_1,\ldots,x_r)$, $y=(y_1,\ldots,y_r)$ and $(z_1,\ldots,z_r)$
be any three elements of $C_1\times\prod\limits_{s\in [r]\setminus\{1\}}\sym{C_s}$ such that
$(x,y,z)$ is a feasible triple in the palette $\PP_1\times\prod\limits_{s=2}^r\sym{\PP_s}$.
For every $s\in[r]$, set $i_s$, $i'_s$, and $i''_s$ to be the three elements of $J_s$
listed in an order such that $\gamma(x)_s=\sgn(i'_s-i_s)$, $\gamma(y)_s=\sgn(i''_s-i_s)$, and $\gamma(z)_s=\sgn(i''_s-i'_s)$.
In particular, it holds that
\begin{enumerate}
\item $i_s<i'_s<i''_s$ if $(\gamma(x)_s,\gamma(y)_s,\gamma(z)_s)=(1,1,1)$,
\item $i'_s<i_s<i''_s$ if $(\gamma(x)_s,\gamma(y)_s,\gamma(z)_s)=(-1,1,1)$,
\item $i_s<i''_s<i'_s$ if $(\gamma(x)_s,\gamma(y)_s,\gamma(z)_s)=(1,1,-1)$,
\item $i''_s<i_s<i'_s$ if $(\gamma(x)_s,\gamma(y)_s,\gamma(z)_s)=(1,-1,-1)$,
\item $i'_s<i''_s<i_s$ if $(\gamma(x)_s,\gamma(y)_s,\gamma(z)_s)=(-1,-1,1)$, and
\item $i''_s<i'_s<i_s$ if $(\gamma(x)_s,\gamma(y)_s,\gamma(z)_s)=(-1,-1,-1)$.
\end{enumerate}
Note that $(\gamma(x)_1,\gamma(y)_1,\gamma(z)_1)=(1,1,1)$ and
$(\gamma(x)_s,\gamma(y)_s,\gamma(z)_s)$ for $s\in [r]\setminus\{1\}$ can only take one of the six values listed above
by the definition of the symmetrization of a palette.
Also note that $i,i',i''\in J_1\times\cdots\times J_r$.
Using the definition of the orders $\preceq_s$, $s\in [r]$, and
the definitions of the symmetrizations of the palettes $\sym{\PP_s}$, $s\in [r]\setminus\{1\}$,
we conclude that
\begin{itemize}
\item $i\prec_si'\prec_si''$, $(\widehat{x_s},\widehat{y_s},\widehat{z_s})\in T_s$ in case (1) above,
\item $i'\prec_si\prec_si''$, $(\widehat{x_s},\widehat{z_s},\widehat{y_s})\in T_s$ in case (2),
\item $i\prec_si''\prec_si'$, $(\widehat{y_s},\widehat{x_s},\widehat{z_s})\in T_s$ in case (3),
\item $i''\prec_si\prec_si'$, $(\widehat{y_s},\widehat{z_s},\widehat{x_s})\in T_s$ in case (4),
\item $i'\prec_si''\prec_si$, $(\widehat{z_s},\widehat{x_s},\widehat{y_s})\in T_s$ in case (5), and
\item $i''\prec_si'\prec_si$, $(\widehat{z_s},\widehat{y_s},\widehat{x_s})\in T_s$ in case (6).
\end{itemize}
In particular, using the definition of the ordered hypergraph $H$,
any three vertices $v\in V_{j_i}$, $v'\in V_{j_{i'}}$ and $v''\in V_{j_{i''}}$ such that
the edge $vv'$ is colored with $\hat x$ in the first coloring,
the edge $vv''$ is colored with $\hat y$, and
the edge $v'v''$ is colored with $\hat z$ form an edge in the hypergraph $H$.
Recall that
\begin{itemize}
\item $h(x)$ is the most frequent color in the second edge-coloring
      among the edges between $V_{j_i}$ and $V_{j_{i'}}$ that are colored with $\hat x$ in the first edge-coloring,
\item $h(y)$ is the most frequent color in the second edge-coloring
      among the edges between $V_{j_i}$ and $V_{j_{i''}}$ that are colored with $\hat y$ in the first edge-coloring, and
\item $h(z)$ is the most frequent color in the second edge-coloring
      among the edges between $V_{j_{i'}}$ and $V_{j_{i''}}$ that are colored with $\hat z$ in the first edge-coloring.
\end{itemize}
Lemma~\ref{lm:triangle} now yields that
there exist vertices $v\in V_{j_i}$, $v'\in V_{j_{i'}}$ and $v''\in V_{j_{i''}}$ such that
the edge $vv'$ is colored with $\hat x$ in the first edge-coloring and with $h(x)$ in the second edge-coloring,
the edge $vv''$ is colored with $\hat y$ in the first edge-coloring and with $h(y)$ in the second edge-coloring, and
the edge $v'v''$ is colored with $\hat z$ in the first edge-coloring and with $h(z)$ in the second edge-coloring.
Since the vertices $v$, $v'$ and $v''$ form an edge in the ordered hypergraph $H$ and $v\preceq_1v'\preceq_1v''$,
it holds that the triple $(h(x),h(y),h(z))$ is a feasible triple in $\PP_0$.
We conclude that $h$ is a homomorphism from the $P_1\times\prod\limits_{s=2}^r\sym{\PP_s}$ to the palette $\PP_0$.
\end{proof}

We next prove the version of Lemma~\ref{lm:single} for multiple palettes.

\begin{lemma}
\label{lm:multi}
Let $\PP_1,\ldots,\PP_r$ and $\PP_0$ be $r+1$ palettes.
Suppose that for every $s\in [r]$,
there exist ordered hypergraphs $H^+_s$ and $H^-_s$ such that
\begin{itemize}
\item $H^+_s$ and $H^-_s$ are $\PP_s$-colorable,
\item both $\unorient{H^+_s}$ and $\unorient{H^-_s}$ are $\PP_{s'}$-colorable for every $s'\in [r]\setminus\{s\}$,
\item $H^+_s$ is not $\PP_0$-colorable, and
\item $H^-_s$ is not $\inv(\PP_0)$-colorable.
\end{itemize}
Then, there exists a hypergraph $H$ that is $\PP_s$-colorable for every $s\in [r]$ but not $\PP_0$-colorable.
\end{lemma}

\begin{proof}
Fix the palettes $\PP_s=(C_s,T_s)$, $s\in [r]\cup\{0\}$, and
the ordered hypergraphs $H^+_s$ and $H^-_s$, $s\in [r]$, as in the statement of the lemma.
Let $c$ be the maximum size of a set $C_i$, $i\in [r]$, and
let $n$ be the maximum number of vertices of one of the hypergraphs $H^+_i$ and $H^-_i$, $i\in [r]$.
Apply Theorem~\ref{thm:ES-grid} with $n$ and $r$
to get $R\in\NN$ with the properties given in the statement of Theorem~\ref{thm:ES-grid}.
Let $N$ be a positive integer such that
\[N!\cdot (1-c^{-rn^2})^{\frac{N^2}{R^{4r}}\cdot\frac{1}{(R^r!)^r}}<1;\]
such $N$ exists since $N!$ is at most $e^{N\log N}$ and the integers $c$, $n$, $r$ and $R$ are fixed.

We will now describe the sought hypergraph $H$, which will have $N$ vertices.
Choose $\preceq_1,\ldots,\preceq_r$ linear orders on $V(H)$ uniformly at random, and
color each pair of vertices of $H$ with an $r$-tuple $(x_1,\ldots,x_r)\in C_1\times\cdots\times C_r$ uniformly at random.
Three vertices $v$, $v'$ and $v''$ form an edge
if the $s$-th coordinates of the colors of their pairs are consistent in the order $\preceq_s$ with the palette $\PP_s$
for every $s\in [r]$;
for example if $v\preceq_2 v''\preceq_2 v'$,
$x_2$ is the second coordinate of the color of the pair $v$ and $v''$,
$y_2$ is the second coordinate of the color of the pair $v$ and $v'$, and
$z_2$ is the second coordinate of the color of the pair $v'$ and $v''$,
then it is required that $(x_2,y_2,z_2)\in T_2$.
This construction ensures that the ordered hypergraph $H^{\preceq_s}$ is $\PP_s$-colorable for every $s\in [r]$, and
so the hypergraph $H$ is $\PP_s$-colorable for all $s\in [r]$.

We will show that the hypergraph $H$ is not $\PP_0$-colorable with positive probability.
Let $F_1,\ldots,F_m$ be any inclusion-wise maximal family of $R^r$-element subsets of $V(H)$ such that
no two sets $F_i$ and $F_j$, $1\le i<j\le m$, have more than a single vertex in common.
Since a single such set blocks less than $\binom{R^r}{2}\binom{N-2}{R^r-2}$ other sets from the inclusion in the family,
the number $m$ of the sets is at least
\[
m\ge\frac{\binom{N}{R^r}}{\binom{R^r}{2}\binom{N-2}{R^r-2}}=\frac{2N(N-1)}{R^{2r}(R^r-1)^2}\ge\frac{N^2}{R^{4r}}.
\]
Index the vertices of each of the sets $F_i$, $i\in [m]$, by $[R]^r$,
i.e., the vertices of $F_i$ are denoted by $v^i_{t_1, \ldots, t_r}$ where $t_1,\ldots,t_r\in [R]$;
the choice of the indexing is arbitrary.
The probability that for fixed $i\in [m]$ and $s\in [r]$,
it holds that $v^i_{t_1, \ldots, t_r}\preceq_s v^i_{t'_1, \ldots, t'_r}$ whenever $t_s<t'_s$,
i.e., the linear order $\preceq_s$ is consistent with the order given by the $s$-th coordinates,
is at least $1/(R^r)!$.
Hence, the probability that this holds for every $s\in [r]$ is at least $1/(R^r)!^r$.
We refer to such a set $F_i$ as \emph{good},
i.e., a set $F_i$ is good
if the indexing of its vertices by the elements of $[R]^r$
is consistent in its $s$-th coordinate with $\preceq_s$ for every $s\in [r]$.
It follows that the expected number of good sets $F_i$, $i\in [m]$, is at least
\begin{equation}
m_0=\frac{N^2}{R^{4r}}\cdot\frac{1}{(R^r)!^r}.
\label{eq:FnumR0}
\end{equation}
Fix any linear orders $\preceq_1,\ldots,\preceq_r$ on $V(H)$ such that the number of good sets is at least $m_0$ (note that
random linear orders $\preceq_1,\ldots,\preceq_r$ have this property with positive probability).

Fix a good set $F_i$ and a linear order $\preceq$ on the vertices of $H$.
We estimate the probability that the subhypergraph of the ordered hypergraph $H^{\preceq}$ induced by $F_i$
is $\PP_0$-colorable
conditioned on the event that the initial $r$ linear orders on $V(H)$ are $\preceq_1,\ldots,\preceq_r$.
By Theorem~\ref{thm:ES-grid},
there exist $n$-element sets $J_1,\ldots,J_r\subseteq [R]$ such that
the order $\preceq$ is the lexicographic order on $J_1\times\cdots\times J_r$
for a suitable permutation of the indices $1,\ldots,r$ and after possibly reversing some of the orders.
In particular, there exists an index $\rho\in [r]$ such that
one of the following the conclusions holds:
\begin{itemize}
\item any two elements $(t_1,\ldots,t_r),(t'_1,\ldots,t'_r)\in J_1\times\cdots\times J_r$ with $t_{\rho}<t'_{\rho}$
      satisfy that $v^i_{t_1,\ldots,t_r}\preceq v^i_{t'_1,\ldots,t'_r}$, or
\item any two elements $(t_1,\ldots,t_r),(t'_1,\ldots,t'_r)\in J_1\times\cdots\times J_r$ with $t_{\rho}<t'_{\rho}$
      satisfy that $v^i_{t_1,\ldots,t_r}\succeq v^i_{t'_1,\ldots,t'_r}$.
\end{itemize}
In other words,
either the order $\preceq$ is consistent with the order given by the $\rho$-th coordinate
when restricted to the subset of $F_i$ indexed by $J_1\times\cdots\times J_r$ or
the order $\preceq$ is consistent with the reverse of this order.
By symmetry (to avoid unnecessarily complex notation),
we assume that $\rho=1$ and that the order $\preceq$ is consistent with the order given by the first coordinate.

Consider the ordered hypergraph $H^+_1$ with the properties given in the statement of the lemma.
Choose $t^w\in\prod\limits_{s=1}^rJ_s$ for each $w\in V(H^+_1)$ such that for every $s\in[r]$, the $s$-th coordinates induce a linear order on the $r$-tuples chosen (i.e., without ties) and for $s=1$, this order  
is consistent with the order of the vertex set of $H^+_1$, while for $s>1$ the order given by the $s$-th coordinates is consistent with an order on the vertex set of $\unorient{H^+_1}$ that
witnesses that $\unorient{H^+_1}$ is $\PP_1$-colorable.
Since the order given by the $s$-th coordinates is consistent with $\preceq_s$ for every $s\in [r]$,
there exists a coloring of the pairs of the vertices $v^i_{t^w}$, $w\in V(H^+_1)$, such that
a copy of $H^+_1$ is contained on the vertices $v^i_{t^w}$, $w\in V(H^+_1)$.
Note that this happens with probability at least
\[\prod_{j=1}^r |C_j|^{-\binom{|V(H^+_1)|}{2}}\ge \prod_{j=1}^r |C_j|^{-n^2} \ge c^{-rn^2},\]
conditioned on the initial $r$ linear orders being $\preceq_1,\ldots,\preceq_r$.
In particular, conditioned on the initial $r$ linear orders being $\preceq_1,\ldots,\preceq_r$,
the restriction of $H^{\preceq}$ to $F_i$ is $\PP_0$-colorable with probability at most $1-c^{-rn^2}$.

Since no two sets $F_i$ and $F_{i'}$, $1\le i<i'\le m$, have more than a single vertex in common,
the events that
the restriction of the ordered hypergraph $H^{\preceq}$ to a good set $F_i$ is $\PP_0$-colorable
are mutually independent. Here we use that the orders $\preceq_s$ are fixed for $s\in[r]$, if they were not fixed they would introduce dependencies.
Hence, using \eqref{eq:FnumR0}, we obtain that for any fixed linear order $\preceq$ on the vertex set, the probability that
the ordered hypergraph $H^{\preceq}$ is $\PP_0$-colorable,
conditioned on the initial $r$ linear orders being $\preceq_1,\ldots,\preceq_r$,
is at most
\[(1-c^{-rn^2})^{\frac{N^2}{R^{4r}}\cdot\frac{1}{(R^r!)^r}}.\]
It follows that the probability that $H$ is $\PP_0$-colorable,
i.e., there exists a linear order $\preceq$ such that $H^{\preceq}$ is $\PP_0$-colorable,
conditioned on the initial $r$ linear orders being $\preceq_1,\ldots,\preceq_r$,
is at most
\[N!\cdot (1-c^{-rn^2})^{\frac{N^2}{R^{4r}}\cdot\frac{1}{(R^r!)^r}},\]
which is less than one by the choice of $N$.
We conclude that the hypergraph $H$ is not $\PP_0$-colorable with positive probability.
\end{proof}

We are now ready to prove our main theorem.

\begin{theorem}
\label{thm:multi}
Let $\PP_1,\ldots,\PP_r$ and $\PP_0$ be $r+1$ palettes.
There exists a hypergraph $H$ that is $\PP_s$-colorable for every $s\in [r]$ but not $\PP_0$-colorable
if and only if
for every $q\in [r]$,
there is no homomorphism from the palette $\PP_q\times\prod\limits_{s\in [r]\setminus\{q\}}\sym{\PP}_s$ to $\PP_0$ or
to $\inv(\PP_0)$.
\end{theorem}

\begin{proof}
Fix the palettes $\PP_1=(C_1,T_1),\ldots,\PP_r=(C_r,T_r)$ and $\PP_0=(C_0,T_0)$.
We start with proving the ``only if'' direction by proving its contrapositive.
Suppose that there exists $q\in [r]$ such that
there is a homomorphism from the palette $\PP_q\times\prod\limits_{s\in [r]\setminus\{q\}}\sym{\PP}_s$ to the palette $\PP_0$ or to $\inv(\PP_0)$.
We write $\PP$ for the palette $\PP_q\times\prod\limits_{s\in [r]\setminus\{q\}}\sym{\PP}_s$ in the rest and
assume that there is a homomorphism from the palette $\PP$ to the palette $\PP_0$ (the other case is symmetric).

Consider a hypergraph $H$ that is $\PP_s$-colorable for every $s\in[r]$.
We show that $H$ is also $\PP$-colorable.
For every $s\in [r]$,
fix a vertex order $\preceq_s$ such that the ordered hypergraph $H^{\preceq_s}$ is $\PP_s$-colorable, and
let $c_s:\binom{V(H)}{2}\to C_s$ be the coloring of the pairs of vertices of $H$ witnessing this.
The ordered hypergraph $H^{\preceq_q}$ is $\sym{\PP}_s$-colorable for every $s\in [r]\setminus\{q\}$:
if $v$ and $v'$ are two vertices of $H$ with $v\preceq_q v'$,
color the pair $v$ and $v'$ with the color $c_s(\{v,v'\})$ if $v\preceq_s v'$, and
color the pair $v$ and $v'$ with its clone $\symc{c_s(\{v,v'\})}$ otherwise.
Observe that
the obtained coloring witnesses that the ordered hypergraph $H^{\preceq_q}$ is $\sym{\PP}_s$-colorable.
It follows that the ordered hypergraph $H^{\preceq_q}$ is $\PP$-colorable.
Proposition~\ref{prop:hom} now yields that $H$ is $\PP_0$-colorable.

It remains to show the ``if'' part of the equivalence in the statement of the theorem.
Assume that there is no homomorphism from the palette $\PP_q\times\prod\limits_{s\in [r]\setminus\{q\}}\sym{\PP_s}$ to $\PP_0$ or to $\inv(\PP_0)$ for any $q\in[r]$.
Lemma~\ref{lm:existm} implies that
there exist ordered hypergraphs $H^+_q$ and $H^-_q$ such that
\begin{itemize}
\item $H^+_q$ and $H^-_q$ are $\PP_q$-colorable,
\item both $\unorient{H^+_q}$ and $\unorient{H^-_q}$ are $\PP_s$-colorable for every $s\in [r]\setminus\{q\}$,
\item $H^+_q$ is not $\PP_0$-colorable, and
\item $H^-_q$ is not $\inv(\PP_0)$-colorable.
\end{itemize}
The existence of a hypergraph that is $\PP_s$-colorable for every $s\in [r]$ but not $\PP_0$-colorable
now follows from Lemma~\ref{lm:multi}.
\end{proof}

Theorem~\ref{thm:multi} immediately yields the following.

\begin{corollary}
\label{cor:multi}
Let $\PP_1,\ldots,\PP_r$ and $\PP'_1,\ldots,\PP'_{r'}$ be $r+r'$ palettes.
There exists a hypergraph $H$ that is $\PP_s$-colorable for every $s\in [r]$ but not $\PP'_s$-colorable for any $s\in [r']$
if and only if
for every $q\in [r]$ and $q'\in [r']$,
there is no homomorphism from the palette $\PP_q\times\prod\limits_{s\in [r]\setminus\{q\}}\sym{\PP}_s$ to $\PP'_{q'}$ or
to $\inv(\PP'_{q'})$.
\end{corollary}

\begin{proof}
We prove the equivalence as two implications.
The ``only if'' part is proven by its contrapositive:
if there exist $q\in [r]$ and $q'\in [r']$ such that
there is a homomorphism from the palette $\PP_q\times\prod\limits_{s\in [r]\setminus\{q\}}\sym{\PP}_s$ to the palette $\PP'_{q'}$ or to $\inv(\PP'_{q'})$,
then every hypergraph $H$ that is $\PP_s$-colorable for all $s\in [r]$ is also $\PP'_{q'}$-colorable by Theorem~\ref{thm:multi}.

We next show the ``if'' part of the equivalence.
For each $q'\in [r']$,
Theorem~\ref{thm:multi} applied with the palettes $\PP_1,\ldots,\PP_r$ and $\PP_0=\PP'_{q'}$
yields that there exists a hypergraph $H_{q'}$ that $\PP_s$-colorable for all $s\in [r]$ but not $\PP'_{q'}$-colorable.
The disjoint union of the hypergraphs $H_1,\ldots,H_{r'}$ is a hypergraph that
is $\PP_s$-colorable for all $s\in [r]$ but not $\PP'_{q'}$-colorable for any $q'\in [r']$.
\end{proof}

\section{Example application}
\label{sec:example}

In this section, we show that there exists a hypergraph with uniform Tur\'an density equal to $4/81$.
We fix three special palettes; see Figures~\ref{fig:PLM}--\ref{fig:P481} for visualizations of the palettes.
The first palette, which is denoted by $\PP_{\rm LM}$, has a single color that is both left and middle.
The set of colors of $\PP_{\rm LM}$ is the set $\{\alpha,\beta',\gamma,\gamma',\omega\}$ and
the palette has two feasible triples: $(\alpha,\omega,\gamma)$ and $(\omega,\beta',\gamma')$.
The second palette, which is denoted by $\PP_{\rm 3T}$ is a chain of three triples.
The set of colors of $\PP_{\rm 3T}$ is the set $\{\alpha,\beta,\beta',\beta'',\gamma'',\omega,\omega'\}$ and
the palette has three feasible triples: $(\alpha,\beta,\omega)$, $(\omega,\beta',\omega')$ and $(\omega',\beta'',\gamma'')$.
The final palette, which is denoted by $\PP_{4/81}$, is a compact version of a chain of two triples.
The set of colors of $\PP_{4/81}$ is the set $\{\alpha,\beta,\gamma,\omega\}$ and
the palette has three feasible triples: $(\alpha,\beta,\gamma)$, $(\alpha,\beta,\omega)$ and $(\omega,\beta,\gamma)$.

\begin{figure}
\begin{center}
\epsfbox{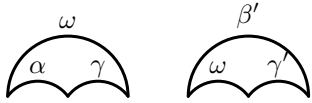}
\end{center}
\caption{The feasible triples of the palette $\PP_{\rm LM}$.}
\label{fig:PLM}
\end{figure}

\begin{figure}
\begin{center}
\epsfbox{hompal-4.mps}
\end{center}
\caption{The feasible triples of the palette $\PP_{\rm 3T}$.}
\label{fig:P3T}
\end{figure}

\begin{figure}
\begin{center}
\epsfbox{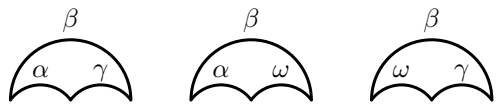}
\end{center}
\caption{The feasible triples of the palette $\PP_{4/81}$.}
\label{fig:P481}
\end{figure}

We first establish three simple lemmas,
starting with showing that
any palette with density larger than $4/81$ admits a homomorphism from $\PP_{\rm LM}$ or $\PP_{\rm 3T}$.

\begin{lemma}
\label{lm:above481}
Let $\PP$ be a palette with $d(\PP)>4/81$.
Then, there exists a homomorphism from $\PP_{\rm LM}$ to $\PP$ or to $\inv(\PP)$, or
there exists a homomorphism from $\PP_{\rm 3T}$ to $\PP$.
\end{lemma}

\begin{proof}
We prove the implication given in the statement of the lemma by contrapositive.
Fix a palette $\PP=(C,T)$ that such that
there is no homomorphism from $\PP_{\rm LM}$ to $\PP$ or to $\inv(\PP)$ and
there is no homomorphism from $\PP_{\rm 3T}$ to $\PP$;
our aim is to show that $d(\PP)\le 4/81$.

Let $M\subseteq C$ be the set of the colors that are the middle color in at least one feasible triple.
Since there is no homomorphism from $\PP_{\rm LM}$ to $\PP$,
none of the colors in $M$ is the left color in any of the feasible triples.
Similarly, since there is no homomorphism from $\PP_{\rm LM}$ to $\inv(\PP)$,
none of the colors in $M$ is the right color in any of the feasible triples.

Let $L\subseteq C$ be the set of colors that are contained as the left color in at least one feasible triple
but are not contained as the right color in any of the feasible triples,
let $R\subseteq C$ be the set of colors that are contained as the right color in at least one feasible triple
but are not contained as the left color in any of the feasible triples, and
let $B\subseteq C$  be the set of colors that are contained as the left color in at least one feasible triple and
as the right color in at least one feasible triple.
The sets $L$, $R$ and $B$ are mutually disjoint by their definition, and
they are also disjoint from the set $M$ 
since any of the feasible triples can contain a color from $M$ as the middle color only.

Observe that there is no feasible triple such that both its left and right color belongs to $B$;
if such a feasible triple existed, then there would be a homomorphism from $\PP_{\rm 3T}$ to $\PP$.
It follows that if $(x,y,z)\in T$, then $x\in L\cup B$, $y\in M$ and $z\in R\cup B$
but $x$ and $z$ are not simultaneously from $B$,
i.e.
\[T\subseteq (L\times M\times R)\cup (L\times M\times B)\cup (B\times M\times R).\]
We can now bound the size of $T$ as follows:
\begin{align*}
|T| & \le |L|\cdot|M|\cdot |R|+|L|\cdot|M|\cdot |B|+|B|\cdot|M|\cdot |R| \\
    & = |M|\left(|L|\cdot |R|+|L|\cdot |B|+|B|\cdot |R|\right) \\
    & \le |M|\cdot 3\cdot \left(\frac{|C|-|M|}{3}\right)^2 \\
    & \le \frac{4|C|^3}{81}.
\end{align*}    
We conclude that the density of the palette $\PP$ is at most $4/81$.
\end{proof}

We next show that the Lagrangian of the palette $\PP_{4/81}$ is equal to $4/81$;
in a certain sense, the argument presented in the proof of Lemma~\ref{lm:Lagrangian481}
illustrates that the palette $\PP_{4/81}$ is extremal in the sense of avoiding
a homomorphism from the palette $\PP_{\rm LM}$ or the palette $\PP_{\rm 3T}$.

\begin{lemma}
\label{lm:Lagrangian481}
The Lagrangian of the palette $\PP_{4/81}$ is equal to $4/81$.
\end{lemma}

\begin{proof}
Observe that the Lagrangian of $\PP_{4/81}$ is the maximum $abc+abd+dbc$
taken over all choices of non-negative $a$, $b$, $c$ and $d$ such that $a+b+c+d=1$,
where $a$ represents the probability of $\alpha$, $b$ that of $\beta$, $c$ that of $\gamma$ and $d$ that of $\omega$.
We now give an upper bound on the value of $abc+abd+dbc$:
\[abc+abd+dbc=b(ac+ad+cd)\le b\cdot 3\cdot \left(\frac{1-b}{3}\right)^2=\frac{b(1-b)^2}{3}\le\frac{4}{81}.\]
It follows that the Lagrangian of $\PP_{4/81}$ is at most $4/81$, and
choosing $a=c=d=2/9$ and $b=1/3$ shows that this upper bound is attained.
\end{proof}

We now state and prove the last lemma that is needed to prove Theorem~\ref{thm:481}.

\begin{lemma}
\label{lm:nohom481}
Neither the palette $\PP_{\rm LM}\times\sym{\PP_{\rm 3T}}$ nor the palette $\PP_{\rm 3T}\times\sym{\PP_{\rm LM}}$
has a homomorphism to the palette $\PP_{4/81}$.
\end{lemma}

\begin{proof}
Observe that the palette $\PP_{\rm LM}\times\sym{\PP_{\rm 3T}}$
contains the following two feasible triples:
\[((\alpha,\alpha),(\omega,\beta),(\gamma,\omega))\mbox{ and }((\omega,\beta),(\beta',\alpha),(\gamma',\symc{\omega}).\]
In particular, the color $(\omega,\beta)$ is the middle color in a feasible triple and
the left color in another feasible triple of the palette $\PP_{\rm LM}\times\sym{\PP_{\rm 3T}}$.
Since there is no such color in the palette $\PP_{4/81}$,
there is no homomorphism from the palette $\PP_{\rm LM}\times\sym{\PP_{\rm 3T}}$ to the palette $\PP_{4/81}$.

Next observe that the palette $\PP_{\rm 3T}\times\sym{\PP_{\rm LM}}$
contains the following feasible triples:
\[
((\alpha,\alpha),(\beta,\omega),(\omega,\gamma))\mbox{, }
((\omega,\gamma),(\beta',\symc{\alpha}),(\omega',\symc{\omega}))\mbox{ and }
((\omega',\symc{\omega}),(\beta'',\symc{\gamma}),(\gamma'',\alpha)).
\]
In particular, 
the palette $\PP_{\rm 3T}\times\sym{\PP_{\rm LM}}$
contains a feasible triple such that
the color $x=(\omega,\gamma)$ is the left color and the color $y=(\omega',\symc{\omega})$ is the right color,
there is another feasible triple where $x$ is the right color, and
there is a third feasible triple where $y$ is the left color.
However, the palette $\PP_{4/81}$ contains no two such (not necessarily distinct) colors $x$ and $y$, and
so there is no homomorphism from the palette $\PP_{\rm 3T}\times\sym{\PP_{\rm LM}}$ to the palette $\PP_{4/81}$.
\end{proof}

We are now ready to show that there exists a hypergraph with uniform Tur\'an density equal to $4/81$.

\begin{theorem}
\label{thm:481}
There exists a hypergraph with uniform Tur\'an density equal to $4/81$.
\end{theorem}

\begin{proof}
As the palettes $\PP_{4/81}$ and $\inv(\PP_{4/81})$ are the same up to renaming colors,
Lemma~\ref{lm:nohom481} yields that
there is no homomorphism from palette $\PP_{\rm LM}\times\sym{\PP_{\rm 3T}}$ to $\PP_{4/81}$ or $\inv(\PP_{4/81})$ and
there is no homomorphism from the palette $\PP_{\rm 3T}\times\sym{\PP_{\rm LM}}$ to $\PP_{4/81}$ or $\inv(\PP_{4/81})$.
Theorem~\ref{thm:multi} implies that
there exists a hypergraph $H$ that is $\PP_{\rm LM}$-colorable and $\PP_{\rm 3T}$-colorable
but not $\PP_{4/81}$-colorable.
We claim that the uniform Tur\'an density of $H$ is equal to $4/81$.

Proposition~\ref{prop:Lagrangian} and Lemma~\ref{lm:Lagrangian481} yield that the uniform Tur\'an density of $H$ is at least $4/81$.
Let $\PP$ be any palette with density larger than $4/81$.
By Lemma~\ref{lm:above481},
there exists a homomorphism from $\PP_{\rm LM}$ to $\PP$ or to $\inv(\PP)$, or
there exists a homomorphism from $\PP_{\rm 3T}$ to $\PP$.
Since the hypergraph $H$ is $\PP_{\rm LM}$-colorable and also $\PP_{\rm 3T}$-colorable,
Proposition~\ref{prop:hom} implies that $H$ is $\PP$-colorable or $\inv(\PP)$-colorable (note that
the latter is equivalent to being $\PP$-colorable).
We conclude that the hypergraph $H$ is $\PP$-colorable for any palette $\PP$ with density larger than $4/81$ and
so with any palette $\PP$ with Lagrangian larger than $4/81$.
Theorem~\ref{thm:palette} now yields that the uniform Tur\'an density of $H$ is equal to $4/81$.
\end{proof}

\section{Conclusion}
\label{sec:concl}

The statement of Theorem~\ref{thm:multi} includes the (somewhat technical) operation of symmetrization of a palette,
which is necessary in this context as we now briefly argue.
Recall that the hypergraph $K_4^{(3)-}$ is $\PP$-colorable and $\inv(\PP)$-colorable
but $K_4^{(3)-}$ is not $\PP\times\inv(\PP)$-colorable,
where $\PP$ is the palette with two colors, $\alpha$ and $\beta$, and
two feasible triples $(\alpha,\beta,\alpha)$ and $(\alpha,\beta,\beta)$.
In particular, setting $\PP_1=\PP$, $\PP_2=\inv(\PP)$ and $\PP_0=\PP\times\inv(\PP)$
demonstrates that
it is not possible to replace the product involving symmetrization
with just the product $\PP_1\times\cdots\times\PP_r$ in the statement of Theorem~\ref{thm:multi}:
the hypergraph $K_4^{(3)-}$ is both $\PP_1$-colorable and $\PP_2$-colorable,
there is a homomorphism from $\PP_1\times\PP_2$ to $\PP_0$ and yet $H$ is not $\PP_0$-colorable.

We now discuss additional settings where the results presented in this paper can be applied.
By identifying the colors of a palette with vertices,
the admissible triples can be seen as oriented edges in a hypergraph,
that is, edges in which the vertices come with a certain order.
With this context, it is possible to use Theorem~\ref{thm:single} and Theorem~\ref{thm:multi}
to relate known results about the Tur\'an density of hypergraphs to the uniform Tur\'an density.
This relation is explored by Wu and the third author in~\cite{LamW25}

One of the results presented in~\cite{LamW25} is the existence of hypergraphs
whose uniform Tur\'an density has arbitrarily high algebraic degree.
The algebraic degree of a real number $\alpha$
is the lowest degree of a nonzero polynomial $p(x)$ with integer coefficients such that $p(\alpha)=0$.
Liu and Pikhurko~\cite{LiuP23} proved that for every $k$, there exists a finite family $\cal F$ of $3$-uniform hypergraphs
whose Tur\'an density has algebraic degree at least $k$.
Combining this result with Theorem~\ref{thm:multi} one can prove the following (we refer to~\cite{LamW25} for particular details):

\begin{theorem}[{Lamaison and Wu~\cite{LamW25}}]
\label{thm:algdegree}
For every $k$, there exists a $3$-uniform hypergraph $H$ whose uniform Tur\'an density has algebraic degree at least $k$.
\end{theorem}

\noindent Note that
unlike in Liu and Pikhurko's result,
Theorem~\ref{thm:algdegree} considers the uniform Tur\'an density of a single hypergraph, not a finite family.

\section*{Acknowledgements}

The first two authors would like to thank Matija Buci\'c for insightful discussions on the uniform Tur\'an density.
The authors would also like to thank the anonymous reviewer for careful reading of the submitted manuscript and
their comments that led to improving the presentation.

\bibliographystyle{bibstyle}
\bibliography{hompal}
\end{document}